\documentclass[12pt, twoside]{article}
\usepackage{amssymb,latexsym}
\usepackage{enumerate}
\usepackage{verbatim}
\usepackage{amsfonts}
\usepackage{amsmath}
\usepackage{fancyhdr}
\usepackage{amsmath,amsxtra,amssymb,amsthm,latexsym,amscd}
\usepackage{array}

\newtheorem{thm}{Theorem}[section]
\newtheorem{cor}[thm]{Corollary}
\newtheorem{lem}[thm]{Lemma}

\newtheorem{prop}[thm]{Proposition}
\newcommand{\ds}{\displaystyle}

\theoremstyle{definition}
\newtheorem{defin}[thm]{Definition}
\newtheorem{rem}[thm]{Remark}
\newtheorem{exa}[thm]{Example}

\numberwithin{equation}{section}

\begin{document}
\date{}
\title{\textsc{\LARGE{some quantitative results on lipschitz inverse and implicit functions theorems}}}
\maketitle
\begin{center}
\textbf{Phan Phien}

\vspace{12pt}
\textit{\small{Department of Natural Sciences}}\\
\textit{\small{Nha Trang College of Education}}\\ 
\textit{\small{1 Nguyen Chanh, Nha Trang, Vietnam}}\\
\textit{\small{e-mail: phieens@yahoo.com}}
\end{center}
\renewcommand{\thefootnote}{}
\footnote{\textit{ }}
\renewcommand{\thefootnote}{\arabic{footnote}}
\setcounter{footnote}{0}
\begin{abstract}
Let $ f: \mathbb{R} ^ n \rightarrow \mathbb{R}^n $ be a Lipschitz mapping with generalized Jacobian at $x_0$, denoted by $\partial f(x_0)$, is of maximal rank. F. H. Clarke (1976) proved that $f$ is locally invertible. In this paper, we give some quantitative assessments for Clarke's theorem on the Lipschitz inverse, and prove that the class of such mappings are open. Moreover, we also present a quantitative form for Lipschitz implicit function theorem.
\end{abstract}
\baselineskip=17pt
\section{Introduction}
Classical inverse and implicit function theorems have attracted many researchers because of their applications in mathematics.  These theorems  are stated for the  class of $C^k$ mappings, and there have been researches for non-smooth mappings and global expansions. We recall some typical results of them.

\vspace*{8pt} \noindent
F. H. Clarke (1976 - \cite{C1}) presented local inverse function theorem for Lipschitz mappings.
Let $ f: \mathbb{R}^n \rightarrow  \mathbb{R}^n $ be a Lipschitz mapping in a neighborhood of $x_0 \in \mathbb{R}^n$. If the generalized Jacobian $ \partial f(x_0)$ at $ x_0 $  is of maximal rank (see Def. \ref{dn2}, \ref{dn3} below), then there exist neighborhoods $ U $ and $ V $ of $ x_0 $ and $ f(x_0) $, respectively, and a Lipschitz function $g: V  \rightarrow  \mathbb{R}^n$ such that

(a) $g(f(u)) = u $ for all $ u \in U$,

(b) $f(g (v)) = v $ for all $ v \in V$.

\vspace * {8pt} \noindent
More general, M. S. Gowda (2004 -\cite{Go}) considered inverse and implicit function theorems for the class of H - differentiable mappings. For the global case,
J. Hadamard (1906 -\cite{H}) presented global diffeomorphism conditions for $ C^1$ mapping class. Generalizing the results of J. Hadamard, P. J. Rabier (1997 -\cite{R}) demonstrated the results for $C^1$ mapping class  on smooth manifolds. O. Gut\'{u} and J. A. Jaramillo (2007 -\cite{Gu-J}) demonstrated global invertible conditions for the class of quasi-isometric mappings between complete metric spaces. Recently, T. Fukui, K. Kurdyka, and L. Paunescu (2010 -\cite{F-K-P}) demonstrated some global inverse function theorems for the class of tame continuous mappings.

\vspace*{8pt} Most of the results on the inverse and implicit function theorems only show the existence of neighborhoods $ U $ and $ V $ to ensure $ f: U \rightarrow V $ is invertible. The quantitative assessments for the subjects in the results were not considered.
It is necessary to use the quantitative assessments for these theorems in several different fields such as: number theory, optimization, theory of measurement, assessment of complex algorithms, ...

\vspace*{8pt} Up to now, in  general case, the problem of quantitative assessment for the classical inverse and implicit function theorems  is unresolved. For the case $ n\leq 2$, P. Henrici (1988 - \cite{He}) gave a quantitative form for analytic function (one variable). Recently, D. Cohen (2005 - \cite{Y}) gave a different proof for the case of analytic functions.  Under the result,  a quantitative form for the theorems in the case of analytic function with two variables was given.

\vspace*{8pt} In this paper, we present a quantitative form for the Clarke inverse function theorem, where $U$, $V$ and the Lipschitz constant of inverse mapping are evaluated quantitatively by  $\partial f (x_0)$. Moreover, we also give a quantitative form for Lipschitz implicit function theorem and prove  that the class of Lipschitz mappings satisfying Clarke's theorem are open: If $f$ is perturbed by a mapping $h$ with  the Lipschitz constant small enough, then the mapping $f + h$ is locally invertible.

The remaining of the paper is organized as follows.
In Section 2 we introduce necessary concepts and results.
Section 3 presents the main results and examples.

\section{Preliminaries}
\subsection{Perturbations and the Inverse}
We give here some definitions, notations and results that will be used later.
\begin{itemize}
\item[] Let $\mathbf{M}_{m\times n}$ denote the vector space of real $m \times n$ matrices,
\item[] $\|x\| =(|x_1|^2 + \cdots + |x_n|^2)^{\frac{1}{2}}$, ~where $x\in \mathbb{R}^n$,\\
$\mathbf{B}^n$ denotes the unit ball in $ \mathbb{R}^n$, $\mathbf{B}_r^n$ denotes the ball of radius $r$, centered at $0\in \mathbb{R}^n$, $\mathbf{B}_r^n(x_0)$ denotes the ball of radius $r$, centered at $x_0\in \mathbb{R}^n$, and $\mathbf{S}^{n-1}$ denote the unit sphere in $\mathbb{R}^n$,
\item[] $\|A\| = \max_{\|x\|=1}\|Ax\|, ~$where $A \in \mathbf{M}_{m\times n}$,
\item[] $\|A\|_F = \ds\left(\sum_{i=1}^m\sum_{j=1}^n\|a_{ij}^2\|\right)^\frac{1}{2},$ where $A$ is a $m\times n$ matrix,
    \item[] If $ A \in \mathbf{M}_{n\times n}$ is an invertible matrix, then
\[\|A^{-1}\|=\frac{1}{\min_{\|x\|=1}\|Ax\|}.\]
\end{itemize}
\noindent Matrix norms have some of the following properties:
\begin{itemize}
\item[(i)] $\|AB\| \leq \|A\|\|B\|$, $\|AB\|_F \leq \|A\|_F\|B\|_F$.
\item[(ii)] $\|A\| \leq \|A\|_F \leq \sqrt{n}\|A\|$, where $A$ is a $m\times n$ matrix.
\item[(iii)] For all $ A \in \mathbf{M}_{m\times n}$ and $ x \in \mathbb{R}^n$, we have $\|Ax\| \leq \|A\| \|x\|$.
\end{itemize}
We topologize $\mathbf{M}_{m\times n}$ with the norm $\|\cdot\|$, and $\mathcal{B}_{m\times n}$ denote the unit ball in $\mathbf{M}_{m\times n}$.
\begin{lem}\label{bd1}
 If $F \in \mathbf{M}_{n\times n}$ and $\|F\|< 1$, then $I - F$ is nonsingular and
  \[(I - F)^{-1} = \sum_{k=0}^\infty F^k\]
 with
 \[\|(I - F)^{-1}\| \leq \frac{1}{1 - \|F\|}.\]
 \end{lem}
\begin{proof} See \cite[Lemma 2.3.3]{G-L}.
\end{proof}
\noindent Based on Lemma \ref{bd1}, we have the following theorem.
\begin{thm}\label{dl1}
Let $A, E \in \mathbf{M}_{n\times n}$. If $A$ is nonsingular and $r = \|A^{-1}E\| < 1$, then $A + E$ is nonsingular and $\|(A + E)^{-1} - A^{-1}\| \leq \|E\|\|A^{-1}\|^2/(1-r)$.
\end{thm}
\begin{proof} See \cite[Theorem 2.3.4]{G-L}.
\end{proof}
\subsection{Generalized Jacobians}
\begin{defin}\label{dn1}
A mapping $ f: \mathbb{R}^m  \rightarrow  \mathbb{R}^n $ is called \textbf{Lipschitz} in a neighborhood of a point $ x_0 $ in $ \mathbb{R}^n $  if there exist a constant $ K $ such that for all $x$ and $y$ near $ x_0 $, we have
\begin{equation}\label{ct1}
\|f(x) - f(y)\|\leq K\|x-y\|.
\end{equation}
\noindent
 If $K \geq 1$ and
 \[\frac{1}{K}\|x-y\| \leq \|f(x) - f(y)\| \leq K \|x-y\|,\]
then $f$ is called \textbf{bi-Lipschitz} or $K$-bi-Lipschitz.
 \end{defin}
 \begin{thm}[Rademacher] \label{dlr} If $f: \mathbb{R}^m \rightarrow \mathbb{R}^n$ is Lipschitz, then $f$ is almost everywhere differentiable.
\end{thm}
\begin{proof}See \cite[Theorem 3.1.6]{F}.
\end{proof}
\noindent The usual $n \times m$ Jacobian matrix of partial derivatives of $f$ at $x$, when it exists, is denoted by $Jf(x)$. By Rademacher's theorem, we have the following definition.
\begin{defin}[F. H. Clarke - \textrm{\cite{C1}, \cite{C2}}]\label{dn2}  The \textbf{Generalized Jacobian} of $f$ at $x_0$, denoted by $\partial f(x_0)$, is the convex hull of all matrices $M$ of the form
\[M = \lim_{i\rightarrow \infty}Jf(x_i),\]
where $x_i$ converges to $x_0$ and $f$ is differentiable at $x_i$ for each $i$.\\
When $f: \mathbb{R}^m \rightarrow \mathbb{R}$, $\partial f(x_0)$ is called the generalized gradient of $f$ at $x_0$.
\end{defin}
\begin{defin} \label{dn3} $\partial f(x_0)$ is said to be of \textbf{maximal rank} if every $M$ in $\partial f(x_0)$ is of maximal rank.
\end{defin}
\begin{rem}
From (\ref{ct1}), $ \partial f(x) $ is bounded in the neighborhood of $x_0$.
\end{rem}
\begin{prop}[\cite{C1}]\label{md1}
$\partial f(x_0)$ is a nonempty compact convex subset of $\mathbf{M}_{n\times m}$.
\end{prop}
\begin{lem}[\cite{C1}] \label{bd22}Let $\varepsilon$ be a positive number. Then for all $x$ sufficiently near $x_0$,
\[\partial f(x) \subset \partial f(x_0) + \varepsilon \mathcal{B}_{n \times m}.\]
\end{lem}
\subsection{Topology of Lipschitz mappings}
Let $f: \mathbb{R}^m \rightarrow \mathbb{R}^n$. Then the Lipschitz constant of $f$ is defined by
\[\textrm{L}(f) = \sup \left\{\frac{\|f(x)-f(y)\|}{\|x-y\|}, ~ x \neq y\right\}.\]
Note that $f$ is Lipschitz if and only if $\textrm{L}(f) < \infty$.\\
Set
$$\textrm{Lip}(\mathbb{R}^m, \mathbb{R}^n) = \left\{f: \textrm{L}(f) < +\infty \right\}.$$
For $f, g \in \textrm{Lip}(\mathbb{R}^m,  \mathbb{R}^n)$ and $\alpha \in \mathbb{R}$, we have the following properties:
\begin{itemize}
\item[(i)] $f+g, \alpha f \in \textrm{Lip}(\mathbb{R}^m,  \mathbb{R}^n)$,
\item[(ii)] $\textrm{L}(f) \geq 0$,
\item[(iii)] $\textrm{L}(f+g) \leq \textrm{L}(f) + \textrm{L}(g)$,
\item[(iv)] $\textrm{L}(\alpha f) = \alpha \textrm{L}(f)$,
\item[(v)] $\textrm{L}(f) = 0 \Leftrightarrow f = \textrm{constant}$.
\end{itemize}

\vspace*{8pt}\noindent
By (v), for $x_0 \in \mathbb{R}^m$, set
\[\textrm{Lip}_{x_0}(\mathbb{R}^m,  \mathbb{R}^n) = \left\{f: f ~\textrm{is Lipschitz  and}~ f(x_0) = 0 \right\}.\]
Then
\[L(f) = 0 \Leftrightarrow f \equiv 0, ~~ \textrm{for all}~ f \in \textrm{Lip}_{x_0}(\mathbb{R}^m,  \mathbb{R}^n).\]
Thus $\textrm{Lip}_{x_0} (\mathbb{R}^m, \mathbb{R}^n) $ is a normed vector space with the norm $\textrm{L}(\cdot)$.
\section{Results - Examples}
Applying the results of F. H. Clarke (\cite{C1}), perturbation matrix, the properties of Lipschitz mappings and differentiable mappings, in Theorem \ref{dl2} we present a quantitative form of the Lipschitz inverse function theorem of F. H. Clarke, in Theorem \ref{dl4} we give a quantitative form of the Lipschitz implicit function theorem, openness of the class of Lipschitz mappings satisfying Clarke's inverse function theorem is proved in Theorem \ref{dl3} and Corollary  \ref{hq31}.
\begin{thm}[c.f.   {\cite[Theorem 1]{C1}}] \label{dl2}
Let $f: \mathbb{R}^n \rightarrow\mathbb{R}^n$ be a Lipschitz mapping with Lipschitz constant $K$. If $~\partial f(x_0)$ is of maximal rank, set
\[\delta = \frac{1}{2}\inf_{M_0\in \partial f(x_0)}\frac{1}{\|M_0^{-1}\|},\]
$r$ be chosen so that $f$ satisfies Lipschitz condition (\ref{ct1})  and
$$\partial f(x) \subset \partial f(x_0) + \delta \mathcal{B}_{n\times n},~~ \textrm{when}~~x \in \mathbf{B}_r^n(x_0),$$ then there exist neighborhoods $U$ and $V$ of  $x_0$ and $f(x_0)$, respectively, and a Lipschitz mapping $~g: V \rightarrow \mathbb{R}^n$ such that

(a) $g(f(u)) = u $ for every $u \in U$,

(b) $f(g(v)) = v $ for every $v \in V$,\\
where,
\[U=\mathbf{B}^n_{\frac{r\delta}{2}.\frac{1}{K}}(x_0), ~~V= \mathbf{B}^n_{\frac{r\delta}{2}}(f(x_0) ),~~\textrm{and}~~L(g) = \ds\frac{1}{\delta}.\]
\end{thm}
First we have:
\begin{lem}[c.f. {\cite[Lemma 3]{C1}}] \label{bd2}
Let $ f: \mathbb{R}^n \rightarrow \mathbb{R}^n $ be a Lipschitz mapping, $\partial f(x_0)$ has maximal rank, set
\[\delta = \frac{1}{2} \inf_{M_0 \in \partial f(x_0)} \frac{1}{\|M_0^{-1} \|},\]
$r$ be chosen so that in $ \mathbf{B}_r^n(x_0)$ $ f $ satisfies the Lipschitz condition (\ref{ct1}) and $ \partial f(x) \subset \partial f(x_0) + \delta \mathcal{B}_{n \times n}$. Then, for every unit vector $v$ in $ \mathbb{R}^n$, there exists a unit vector $w$ in $ \mathbb{R}^n $ such that, whenever $ x$ lies in $x_0 + r \mathbf{B}^n$ and $M$ belongs to $ \partial f(x) $,
\begin{equation}\label{ct4}w.(Mv) \geq \delta.
\end{equation}
\end{lem}
\begin{proof}
By Proposition \ref{md1} and $ \partial f (x_0) $ is of maximal rank, the subset $ \partial f(x_0) \mathbf{S}^{n-1}$ of $ \mathbb{R}^n$ is compact and not containing $0$. For $M_0 \in \partial f(x_0)$, we have
\[\min_{\|x\|=1}\|M_0x\| = \frac{1}{\|M_0^{-1}\|}.\]
Set
\[\delta = \frac{1}{2}\inf_{M_0\in \partial f(x_0)}\frac{1}{\|M_0^{-1}\|},\]
we get $ \partial f(x_0) \mathbf{S}^{n-1}$ distances $2 \delta $ from $0$. \\
If $M \in G = \partial f(x_0) + \varepsilon \mathcal{B}_{n\times n}$, then
\[\min_{\|x\|=1}\|Mx\| \geq \min_{\|x\|=1}\|M_0x\| - \varepsilon.\]
Choosing
\[\varepsilon = \delta = \frac{1}{2}\inf_{M_0\in \partial f(x_0)}\frac{1}{\|M_0^{-1}\|},\]
we get $G \mathbf{S}^{n-1}$ distances at least $\delta$ from $0$. By Lemma \ref{bd22}, there exists a positive number $r$, such that
\begin{equation} \label{ct5}
x \in x_0 + r\mathbf{B}^n \Rightarrow \partial f(x) \subset G.
\end{equation}
Let $r$ be chosen so that $f$ satisfies (\ref{ct1}) on $x_0 + r \mathbf{B}^n$. \\
Thus, let any unit vector $v$ be given, apply the above results, the convex set $Gv$ distances at least $\delta$ from $0$. By the usual separation theorem for convex sets, there exists a unit vector $w$ such that
\[w.(\gamma v) \geq \delta,\]
for every $\gamma \in G$.
Hence, applying (\ref{ct5}) we obtain (\ref {ct4}).
\end{proof}
\begin{proof}[Proof of  Theorem \ref{bd1}]
Using the proof of  \cite[Theorem 1]{C1}  to replace \cite[Lemma 3]{C1}  by  the preceding lemma.\\
\textbf{Estimation of the neighborhood $U$, $V$ and the Lipschitz constant $L(g)$}: \\
According to the proof of \cite[Theorem 1]{C1}, we have
\[L(g) = \ds\frac{1}{\delta},\]
\[V = f(x_0) + (r\delta/2)\mathbf{B}^n=\mathbf{B}^n_{\frac{r\delta}{2}}\left(f(x_0)\right),\]
and choose $U$ being an arbitrary neighborhood of $x_0 $ and satisfying $f(U) \subset V$. Then for all $ x \in U $ we have
\[\|f(x) - f(x_0)\| \leq K\|x-x_0\|\leq \frac{r\delta}{2}.\]
Hence,
\[\|x-x_0\|\leq \frac{r\delta}{2}\frac{1}{K}.\]
So
\[U = \mathbf{B}^n_{\frac{r\delta}{2}\frac{1}{K}}\left(x_0\right).\]
\end{proof}
\begin{rem} When $f$ is $C^1$, $\partial f(x_0)$ reduces to $ Jf(x_0) $, and $g$ is in the class $C^1$. Thus we get the quantitative form of the classical inverse function theorem.
\end{rem}
\begin{rem}If $F: A \rightarrow \mathbb{R}^n$ be a Lipschitz mapping in a neighborhood of $(x_0, y_0)$, $A = U \times V$ be a open subset of $\mathbb{R}^m \times \mathbb{R}^n$, then the generalized Jacobian of $F$ at $(x_0, y_0)$ satisfies
\[\partial F(x_0, y_0) \subset \left\{\left(
                                  \begin{array}{cc}
                                    M_1 & M_2 \\
                                  \end{array}
                                \right): M_1 \in \partial_1 F(x_0, y_0), M_2 \in \partial_2 F(x_0, y_0)\right\},\]
where $\partial_1 F(x_0, y_0) ~\textrm{and}~  \partial_2 F(x_0, y_0)$ are generalized Jacobians of $F(\cdot, y_0): U \rightarrow \mathbb{R}^n  ~\textrm{and}~ F(x_0, \cdot): U \rightarrow \mathbb{R}^n$ at $(x_0, y_0)$, respectively.
\end{rem}
\begin{thm} \label{dl4}
Let $F: A \rightarrow \mathbb{R}^n$ be a Lipschitz mapping in a neighborhood of $(x_0, y_0)$ with Lipschitz constant $K$, $A = U \times V$ be a open subset of $\mathbb{R}^m \times \mathbb{R}^n$. Suppose that $\partial_2F(x_0, y_0)$ is of maximal rank, $F(x_0, y_0) = 0$, set
\[\delta = \frac{1}{2} \inf_{M_2 \in \partial_2 F(x_0, y_0)}\frac{1}{(m+(1+mK^2)n\| M_2^{-1}\|^2)^\frac{1}{2}},\]
 $r$ be chosen so that  $F$ satisfies Lipschitz condition (\ref{ct1}) and $$ \partial F (x, y) \subset \partial F(x_0, y_0) + \delta \mathcal{B}_{n\times(m+n)}, ~\textrm{when}~ (x, y) \in  \mathbf{B}_r^{m+n}((x_0, y_0)).$$
Then there exist a Lipschitz mapping $g: U_0 \rightarrow V$ defined in a neighborhood $U_0 \subset \mathbb{R}^m$ of $x_0 $ such that $g(x_0) = y_0$ and
\[F(x, g (x)) = 0,\]
for all $ x \in U_0$. \\
Moreover,
\[U_0 = \mathbf{B}^m_{\frac{r  \delta}{2} \frac{1}{K +1} }\left (x_0 \right ), ~\textrm{and}~~ L(g) \leq  \sup_{M_2 \in \partial_2 F(x_0, y_0)} K \|M_2^{-1}\|.\]
 \end{thm}
 \begin{proof}\textit{ }\\
\textbf{ 1.} Set $f(x, y) = (x, F (x, y)), ~ \textrm{for all}~ (x, y) \in U \times V$.
 Since $F$ is Lipschitz, according to Radamacher's theorem, there exists the generalized Jacobian of $F$ at $(x_0, y_0)$.
In a neighborhood of $(x_0, y_0)$, $F$ is almost everywhere differentiable, therefore existing generalized Jacobian of $f$ at $(x_0, y_0)$ and
\[\partial f(x_0, y_0) \subset \left\{\left(\begin{array}{cc}
                                       I_m & 0 \\
                                       M_1 & M_2
                                     \end{array}
\right): \ \ M_1 \in \partial_1 F(x_0, y_0), \ \ M_2 \in \partial_2 F(x_0, y_0) \right\}.\]
Then $\partial f(x_0, y_0)$ is of maximal rank, because $\partial_2 F(x_0, y_0)$ is of maximal rank.

\vspace*{8pt} \noindent
\textbf{2.} For $M = \left(\begin{array}{cc}
                                       I_m & 0 \\
                                       M_1 & M_2
                                     \end{array}
\right) \in \partial f(x_0, y_0)$,
\[M^{-1} = \left(
             \begin{array}{cc}
               I_m & 0 \\
               -M_2^{-1}M_1 & M_2^{-1} \\
             \end{array}
           \right).
\]
Therefore, we have
\[\ds\begin{array}{rcl}
    \|M^{-1}\| \leq \|M^{-1}\|_F & = & \left(m + \|M_2^{-1}M_1\|_F^2 + \|M_2^{-1}\|_F^2\right)^\frac{1}{2} \\
     & \leq & \left(m + \|M_2^{-1}\|_F^2( \|M_1\|_F^2 + 1)\right)^\frac{1}{2} \\
     & \leq & \left(m + \|M_2^{-1}\|_F^2( m\|M_1\|^2 + 1)\right)^\frac{1}{2} \\
     & \leq & \left(m + (1 + mK^2) \|M_2^{-1}\|_F^2\right)^\frac{1}{2}.
  \end{array}
\]
Thus
\[\frac{1}{\|M^{-1}\|} \geq \frac{1}{\left(m + (1 + mK^2)\|M_2^{-1}\|_F^2\right)^\frac{1}{2}}\geq \frac{1}{\left(m + (1 + mK^2)n\|M_2^{-1}\|^2\right)^\frac{1}{2}}.\]
Set
\[\Delta = \frac{1}{2}\inf_{M\in \partial f(x_0, y_0)}\frac{1}{\|M^{-1}\|},\]
we get
\begin{equation}\label{ct341}\Delta \geq \delta.
\end{equation}

\vspace*{8pt} \noindent
\textbf{3.} According to the theorem, we have
$$ \partial F (x, y) \subset \partial F(x_0, y_0) + \delta \mathcal{B}_{n\times (m+n)}, ~\textrm{when}~ (x, y) \in  \mathbf{B}^{m+n}_r((x_0, y_0)),$$
by (\ref{ct341}), we get
$$ \partial F (x, y) \subset \partial F(x_0, y_0) + \Delta  \mathcal{B}_{n\times (m+n)}, ~\textrm{when}~ (x, y) \in  \mathbf{B}^{m+n}_r((x_0, y_0)).$$
Therefore, we can chose $r$ so that $f$ satisfies Lipschitz condition (\ref{ct1}) and
\[\partial f(x, y) \subset \partial f(x_0, y_0) + \Delta \mathcal{B}_{(m+n)\times(m+n)}, ~\textrm{when}~ (x, y) \in  \mathbf{B}^{m+n}_r((x_0, y_0)).\]

\vspace*{8pt} \noindent
\textbf{4.} Since $F$ is a Lipschitz mapping with coefficient $K$, $f$ is Lipschitz with coefficient $K +1$. Applying Theorem \ref{dl2}, $f$ is locally invertible and
\[f^{-1}(x, z) = (x, h(x, z)),\]
with $h$ is a Lipschitz mapping. Define
\[g(x) = h(x, 0).\]
Then $g$ is Lipschitz and
\[(x, F(x, g (x))) = f(x, g(x)) = f(x, h(x, 0)) = f (f ^{-1}(x, 0 )) = (x, 0).\]
This indicates the existence of $g$ satisfying the requirements of the theorem.

\vspace*{8pt} \noindent
\textbf{5.} Estimation of the neighborhood $U_0 $ of $ x_0 $ and $L(g)$:\\
Applying Theorem \ref{dl2},  we obtain
\[(x, h(x, z)) \in U= \mathbf{B}^{m+n}_{\frac{r \Delta}{2} \frac{1}{k +1}} \left((x_0, y_0)\right).\]
Thus
\[(x, g(x)) \in U'= \mathbf{B}^{m+n}_{\frac{r \Delta}{2} \frac{1}{k +1}}\left((x_0, 0)\right).\]
Hence, by (\ref{ct341}), the theorem is satisfied for all $(x, g(x)) \in U'' = \mathbf{B}^{m+n}_{\frac{r \delta}{2} \frac{1}{k +1}}\left((x_0, 0)\right).$
Projecting $U''$ onto the space $ \mathbb{R}^m$, we get
\[U_0 = \mathbf{B}^m_{\frac{r  \delta}{2} \frac{1}{k +1}} \left (x_0\right ).\]
Moreover, applying the formula of implicit function derivative,
\[Dg = -\left( \frac{\partial F}{\partial y}\right)^{-1} \frac{\partial F}{\partial x}, ~~\textrm{whenever}~~ \left( \frac{\partial F}{\partial y}\right)^{-1}~~\textrm{exist}. \]
Therefore, we get
\[L(g) \leq  \sup_{M_2 \in \partial_2 F(x_0, y_0)} K \|M_2^{-1}\|.\]
 \end{proof}
\begin{rem} When $F$ is $C^1$, $\partial_2 F(x_0, y_0)$ reduces to $ J_2F(x_0, y_0) $, and $g$ is in the class $C^1$. Thus we get the quantitative form of the classical implicit function theorem.
\end{rem}
\begin{exa}
For $m = 1, n=2$, consider $F(x, y, z) = (2x+|y|+3y, 2x+|z|+3z)$ in $\mathbf{B}^3\left((0, 0, 0)\right)$. Then
\[\|F(x, y, z) - F(x', y', z')\| \leq \sqrt{24}\|(x, y, z) - (x', y', z')\|.\]
Thus $F$ is Lipschitz with Lipschitz constant $K = \sqrt{24}$.\\
We have
\[J_2F(x_i, y_i, z_i) = \left(
                   \begin{array}{cc}
                     \frac{|y_i|}{y_i}+3 & 0 \\
                     0 &  \frac{|z_i|}{z_i}+3  \\
                   \end{array}
                 \right),~~ (x_i, y_i, z_i) ~\textrm{near}~(0, 0, 0).
\]
Hence,
\[\partial_2F(0, 0, 0)=\left\{\left(
                            \begin{array}{cc}
                              s+3 & 0 \\
                              0 & t+3 \\
                            \end{array}
                          \right): -1 \leq s \leq 1, -1 \leq t \leq 1
\right\},\]
and $\partial_2F(0, 0, 0)$ is of maximal rank.\\
Let $M_2 \in \partial_2F(0, 0, 0)$. Then there exist $M_2^{-1}$ defined by
\[M_2^{-1} = \left(
                            \begin{array}{cc}
                              \frac{1}{s+3} & 0 \\
                              0 &  \frac{1}{t+3} \\
                            \end{array}
                          \right).\]
So
\[\delta = \frac{1}{2} \inf_{M_2 \in \partial_2 F(x_0, y_0)}\frac{1}{(m+(1+mK^2)n\| M_2^{-1}\|^2)^\frac{1}{2}} = \frac{1}{\sqrt{54}}.\]
By the preceding theorem, there exist a Lipschitz mapping $g$ such that $g(0) = (0, 0)$ and
\[F(x, g(x)) = (0, 0), ~~\textrm{for all}~x \in U_0.\]
Moreover, for $(x, y, z)$ near $(0, 0, 0)$, we have
\[JF(x, y, z) = \left(
                   \begin{array}{ccc}
                    2& \frac{|y|}{y}+3 & 0 \\
                    2& 0 &  \frac{|z|}{z}+3  \\
                   \end{array}
                 \right).\]
                 We can chose $r = 1$, and then
               \[ \partial F(x, y, z) \subset \partial F(0, 0, 0) + \frac{1}{\sqrt{54}} \mathcal{B}_{2\times 3}, ~\textrm{every}~ (x, y, z) \in  \mathbf{B}^3_r((0,0, 0)).\]
Hence, we obtain
\[U_0 = \mathbf{B}^1_{\frac{1}{6\sqrt{6}}\cdot \frac{1}{1+2\sqrt{6}}}\left(0 \right), ~\textrm{and}~ L(g) \leq \sqrt{6}.\]
\end{exa}
\begin{thm}\label{dl3}
Let $f_0: \mathbb{R}^n \rightarrow \mathbb{R}^n$ be a Lipschitz mapping in the neighborhood of $x_0$ so that $\partial f_0(x_0)$ is of maximal rank and satisfies
\[K \| x-y \| \leq \| f_0(x)-f_0(y) \| \leq K' \| x-y \|.\]
Let $f = f_0 + h$, with $ h: \mathbb{R}^n \rightarrow \mathbb{R}^n$ is a Lipschitz mapping with Lipschitz constant L so that
\[L < K.\]
Set
\[\delta = \frac{1}{2} \inf_{M_0 \in \partial f(x_0)} \frac{1}{\| M_0^{-1}\|}.\]
Suppose that  $~r$ was chosen so that  $f$ satisfies Lipschitz condition (\ref{ct1}) and $\partial f(x) \subset \partial f(x_0) + \delta \mathcal{B}_{n\times n}$ in $\mathbf{B}^n_r(x_0)$.
Then there exist neighborhoods $U$ and $V$ of $x_0$ and $f(x_0)$, respectively, and a Lipschitz mapping $g: V \rightarrow \mathbb{R}^n$ such that

\vspace*{8pt}
(a) $g(f(u)) = u$ for every $u \in U$,

\vspace {8pt}
(b) $f(g(v)) = v $ for every $v \in V$.\\
Moreover, $U, V$ and $L(g)$ are determined by
\[U = \mathbf{B}^n_{\frac{r \delta}{2} \frac{1}{K '+ L} }\left(x_0\right ),~~~~ V = \mathbf{B}^n_{\frac{r \delta}{2}} \left (f(x_0)\right), ~~\textrm{and}~~L(g) = \ds\frac{1}{\delta}.\]
\end{thm}
\begin{proof}\textit{ }\\
\textit{Remark}: $f$ is differentiable at $x_i$ if and only if $f$ can be approximated in a neighborhood of $x_i$ by the affine mapping
\[T(x) = f(x_i) + Jf(x_i)(x-x_i).\]
Hence,
\[K\|x-x_i\| \leq \|f(x) - f(x_i)\| \approx \|Jf(x_i)(x-x_i)\| \leq K'\|x-x_i\|.\]
From which we obtain
\begin{equation}\label{ct2}K \leq \inf_{y\neq 0}\frac{\|Jf(x_i)y\|}{\|y\|}, \forall Jf(x_i) \Rightarrow K \leq \inf_{Jf(x_i)}\left(\min_{\|x\|=1}\|Jf(x_i)x\|\right),
\end{equation}
\begin{equation}\begin{array}{ccl}
                  \sup_{y\neq 0}\frac{\|Jf(x_i)y\|}{\|y\|}\leq K', \forall Jf(x_i) & \Rightarrow & \sup_{Jf(x_i)}\left(\max_{\|x\|=1}\|Jf(x_i)x\| \right)\leq K' \\
                  \\
                  & \Rightarrow & \sup_{Jf(x_i)}\|Jf(x_i)\| \leq K'.
                \end{array}
    \label{ct3}
\end{equation}

\vspace*{8pt}
\noindent
From the remarks above we will prove that $f$ is a Lipschitz mapping and  $\partial f(x_0)$ is of maximal rank. \\
We have
\[\begin{array}{rcl}
    \|f(x)-f(y)\| & = & \|(f_0(x)-f_0(y))+(h(x)-h(y))\| \\
     & \leq & \|(f_0(x)-f_0(y))\|+\|(h(x)-h(y))\| \\
     & \leq & (K'+L)\|x-y\|.\\
     \end{array}
\]
So $f$ is Lipschitzian with Lipschitz constant $K'+L$.

\vspace*{8pt}
\noindent
According to Rademacher's theorem, $f$ is almost everywhere differentiable near $x_0$, so that the general Jacobian $ \partial f(x_0)$ exits. Moreover,
\[\partial f(x_0) \subset \partial f_0(x_0) + \partial h(x_0).\]
Indeed, let $E$ be the set of points where $f_0$ or $h$ is fail to be differentiable. Then every $M \in \partial f(x_0), M$ has the form
\[M = \lim_{i\rightarrow \infty}J(f_0+h)(x_0+h_i), h_i \rightarrow 0 ~\textrm{when}~i \rightarrow \infty,\]
here the sequence $ \{x_0 + h_i \} $ lies  in the complement of $E$, and admits a subsequence $\{x_0 + h_{n_i}\}$ such that $Jf_0(x_0 + h_{n_i})$ and $Jh(x_0 + h_{n_i})$ both exist and converge. Hence
 \[M= \lim_{i\rightarrow \infty}Jf_0(x_0+h_{n_i}) + \lim_{i\rightarrow \infty}Jh(x_0+h_{n_i}) = M_0+H,\]
where $M_0 \in \partial f_0(x_0), H \in \partial h(x_0)$.

\vspace*{8pt}
\noindent
Next, we prove $ \partial f(x_0) $ is of maximal rank: \\
Using (\ref{ct2}) and (\ref{ct3}) we have
\[\ds\begin{array}{rcrcl}
    L<K & \Rightarrow & \sup_{H\in \partial h(x_0)}\|H\| & < & \inf_{M_0\in\partial f_0(x_0)}\left(\min_{\|x\|=1}\|M_0x\|\right) \\
     & \Rightarrow & \sup_{H\in \partial h(x_0)}\|H\| & < &\inf_{M_0\in\partial f_0(x_0)} \left( \frac{1}{\frac{1}{\min_{\|x\|=1}\|M_0x\|}}\right) \\
          & \Rightarrow & \sup_{H\in \partial h(x_0)}\|H\| & < &\frac{1}{ \sup_{M_0\in\partial f_0(x_0)} \left(\frac{1}{\min_{\|x\|=1}\|M_0x\|}\right)} \\
    & \Rightarrow & \sup_{H\in \partial h(x_0)}\|H\| & < & \frac{1}{\sup_{M_0\in\partial f_0(x_0)}\|M_0^{-1}\|},
     \end{array}
\]
so we get
\[\sup_{H\in \partial h(x_0), M_0\in \partial f_0(x_0)}\|M_0^{-1} H\| < 1. \]
According to Theorem \ref{dl1}, $M_0 + H$ is of maximal rank for all $H\in \partial h(x_0), M_0\in \partial f_0(x_0)$.\\
According to the proof above,  if $M \in \partial f(x_0)$ then $M = M_0 + H$, with $M_0\in \partial f_0(x_0), H\in \partial h(x_0)$. Hence, $M$ is of maximal rank for every $M\in \partial f(x_0)$.\\
So $\partial f(x_0)$ is of maximal rank.

\vspace*{8pt}
\noindent
Thus $f$ is Lipschitz and $ \partial f(x_0)$ is of maximal rank, applying Theorem \ref{dl2}, we get the results of the theorem.

\vspace*{8pt}
\noindent
Moreover, according to the proof of Theorem \ref{dl2}, we have
\[L(g) = \frac{1}{\delta}, \ \ \ \ V = f(x_0) + (r\delta/2)\mathbf{B}^n=\mathbf{B}^n_{\frac{r\delta}{2}}\left(f(x_0)\right),\]
and choose $U$ being an arbitrary neighborhood of $x_0 $ and satisfying $f(U) \subset V$. Then for all $ x \in U $ we have
\[\|f(x) - f(x_0)\| \leq (K'+L)\|x-x_0\|\leq \frac{r\delta}{2}.\]
Hence,
\[\|x-x_0\|\leq \frac{r\delta}{2}\frac{1}{K'+L}.\]
So
\[U = \mathbf{B}^n_{\frac{r\delta}{2}\frac{1}{K'+L}}\left(x_0\right).\]
\end{proof}
\begin{rem}If $ L> K $, then $f = f_0 + h $ may not satisfy local invertible condition.
\end{rem}
\begin{cor}\label{hq31} The class of Lipschitz mappings satisfying Clarke's inverse function theorem  is open in the space
$\textrm{Lip}_{x_0}(\mathbb{R}^m,  \mathbb{R}^n)$.
\end{cor}
\begin{exa}
For $n=2$, consider $f_0(x, y) = (|x|+2x, |y|+2y)$ in $\mathbf{B}^2\left((0, 0)\right)$. We have
\[Jf_0(x_i, y_i) = \left(
                   \begin{array}{cc}
                     \frac{|x_i|}{x_i}+2 & 0 \\
                     0 &  \frac{|y_i|}{y_i}+2  \\
                   \end{array}
                 \right),~~ (x_i, y_i) ~\textrm{near}~(0, 0).
\]
Thus
\[\partial f_0(0, 0)=\left\{\left(
                            \begin{array}{cc}
                              s+2 & 0 \\
                              0 & t+2 \\
                            \end{array}
                          \right): -1 \leq s \leq 1, -1 \leq t \leq 1
\right\},\]
and $\partial f_0(0, 0)$ is of maximal rank.\\
We have
\[\begin{array}{ccl}
    \|f_0(x, y) - f_0(x', y')\| & = & \|(|x|+2x, |y|+2y)-(|x'|+2x', |y'|+2y')\| \\
    & = & \|((|x|-|x'|)+2(x-x'), (|y|-|y'|)+2(y-y'))\|.
  \end{array}
\]
Hence,
\[\|(x, y) - (x', y')\| \leq \|f_0(x, y) - f_0(x', y')\| \leq 3\|(x, y) - (x', y')\|. \]
Thus, $f_0$ is Lipschitz with Lipschitz constants $K = 1, K' = 3$.

\vspace*{8pt}
\noindent \textbf{1. Case $L<K$.}
We find $h(x, y) = (\frac{1}{2}|x|, \frac{1}{2}|y|)$ is Lipschitz with Lipschitz constant $L = \frac{1}{2}$ satisfying
\[L < K.\]
Set $f = f_0+h$, we have
\[f(x, y) =( \frac{3}{2}|x|+2|x|, \frac{3}{2}|y|+2y),\]
and $f$ is Lipschitz in $\mathbf{B}^2\left((0, 0)\right)$ with Lipschitz constant $\frac{7}{2}.$\\
Then,
\[Jf(x_i, y_i) = \left(
                   \begin{array}{cc}
                    \frac{3}{2} \frac{|x_i|}{x_i}+2 & 0 \\
                     0 &  \frac{3}{2}\frac{|y_i|}{y_i}+2  \\
                   \end{array}
                 \right),~~ (x_i, y_i) ~\textrm{near}~(0, 0),
\]
\[\partial f(0, 0)=\left\{\left(
                            \begin{array}{cc}
                               s+2 & 0 \\
                              0 &  t+2 \\
                            \end{array}
                          \right): -\frac{3}{2} \leq s \leq \frac{3}{2}, -\frac{3}{2} \leq t \leq \frac{3}{2}
\right\},\]
and $\partial f(0, 0)$ is of maximal rank.\\
So $f = f_0+h$ is locally invertible.\\
We have
\[\delta = \frac{1}{2} \inf_{M_0 \in \partial f(x_0)} \frac{1}{\| M_0^{-1}\|} = 1.\]
Moreover,
\[Jf(x, y) = \left(
                   \begin{array}{cc}
                    \frac{3}{2} \frac{|x|}{x}+2 & 0 \\
                     0 &  \frac{3}{2}\frac{|y|}{y}+2  \\
                   \end{array}
                 \right),~~ (x, y) \in \mathbf{B}^2_r\left((0, 0)\right), ~ r \leq 1.
\]
We can chose $r = 1$, and then
\[\partial f(x, y) \subset \partial f(0, 0) + \mathcal{B}_{2\times 2}, ~~\textrm{when} ~~(x, y) \in \mathbf{B}^2_r\left((0, 0)\right).\]
Hence, we obtain
\[U = \mathbf{B}^2_{\frac{1}{7}} \left((0, 0)\right ),~~~~ V = \mathbf{B}^2_{\frac{1}{2}}\left ((0, 0)\right), ~~\textrm{and}~~L(g) = 1.\]

\vspace*{8pt}
\noindent \textbf{2. Case $L>K$.}
We find $h(x, y) = (\frac{3}{2}|x|, \frac{3}{2}|y|)$ is Lipschitz with Lipschitz constant $L = \frac{3}{2}>K$.
Set $f = f_0+h$, we have
\[f(x, y) =( \frac{5}{2}|x|+2x, \frac{5}{2}|y|+2y),\]
and $f$ is Lipschitz with Lipschitz constant $\frac{9}{2}.$\\
Then,
\[Jf(x_i, y_i) = \left(
                   \begin{array}{cc}
                    \frac{5}{2} \frac{|x_i|}{x_i}+2 & 0 \\
                     0 &  \frac{5}{2}\frac{|y_i|}{y_i}+2  \\
                   \end{array}
                 \right),~~ (x_i, y_i) ~\textrm{near}~(0, 0),
\]
\[\partial f(0, 0)=\left\{\left(
                            \begin{array}{cc}
                              s+2 & 0 \\
                              0 & t +2\\
                            \end{array}
                          \right): -\frac{5}{2} \leq s \leq \frac{5}{2}, -\frac{5}{2} \leq t \leq \frac{5}{2}
\right\},\]
and $\partial f(0, 0)$ is not of maximal rank.
\end{exa}
\noindent {\large{\textbf{Acknowledgements.}}}

\vspace*{8pt}\noindent
The author wishes to thank Professor Ta Le Loi for the suggestion of writing down this paper.
This research is supported  by Vietnam's
National Foundation for Science and Technology Development
(NAFOSTED).


\begin{thebibliography}{10}
\baselineskip=14pt
\bibitem{C1} F. H. Clarke, {\it On the inverse function theorem,}
Pacific Journal of Mathematics, Vol 64, No 1 (1976), 97-102.
\bibitem{C2} F. H. Clarke, {\it Generalized gradients and applications,}
Trans. Amer. Math. Soc., Vol 205 (1975), 247-262.
\bibitem{C-L-S-W} F. H. Clarke, Yu.S. Ledyaev, R.J. Stern and P.R. Wolenski, ``Nonsmooth Analysis and Control Theory'', Graduate Texts in Mathematics, Springer-Verlag, New York, (1998).
    \bibitem{F}  H. Federer, ``Geometric measures theory'',
Springer-Verlag, (1969).
\bibitem{F-K-P}  T. Fukui, K. Kurdyka, and L. Paunescu, {\it Tame Nonsmooth Inverse Mapping Theorems,}
SIAM Journal On Optimation. Volume 20, Issue 3 (2010), 1573-1590.
\bibitem{Go} M. S. Gowda, {\it Inverse and Implicit Function Theorems for H-Differeniable and Semismooth Functions,} Optimization Methods and Software, Vol. 19 (2004), 443-461.
\bibitem{Gu-J} O. Gut\'{u} and J. A. Jaramillo, {\it Global homeomorphisms and covering projections on metric spaces,} Math. Ann. (2007), 338: 75-95.
\bibitem{G-L} G. H. Golub and C. F. van Loan , ``Matrix computation'', Johns Hopkins Univ. Press (1983).
    \bibitem{H} J. Hadamard, \textit{Sur les transformations ponctuelles}, Bull. Soc. Math. Fr. 34 (1906), pp. 71-84.
        \bibitem{He} P. Henrici, ``Applied and Computational Complex Analysis'', Vol. 1, Wiley, New York, (1988).
    \bibitem{R} P. J. Rabier, \textit{Ehresmann fibrations and Palais-Smale conditions for morphisms of Finsler manifors}, Annals of Mathematics, 146 (1997), 647-691.
    \bibitem{W} N. Weaver, ``Lipschitz Algebras'', Uto-Print Singapore, (1999).
    \bibitem{Y} Y. Yomdin , \textit{Some quantitative results in singularity theory}, Anales Polonici Mathematici, 37 (2005), 277-299.
\end{thebibliography}
\end{document}